\documentclass[journal]{IEEEtran}
\usepackage{amsmath,amssymb,amsfonts}

\usepackage{algorithm,algorithmic}
\usepackage{graphicx}
\usepackage{textcomp}

\usepackage{color}

\usepackage{lipsum}
\usepackage{epstopdf}

\usepackage{amsopn}

\usepackage{amsthm}

\usepackage{enumerate}
\usepackage{cancel}
\usepackage{mathrsfs}
\usepackage{mathdots}
\usepackage{euscript}
\usepackage{amscd}
\usepackage{cite}
\usepackage{placeins}
\usepackage{tikz}
\usetikzlibrary{snakes,arrows,shapes}

\graphicspath{{images/}}

\newtheorem{proposition}{Proposition}
\theoremstyle{definition}

\theoremstyle{definition}
\theoremstyle{remark}\newtheorem{remark}{Remark}

\newcommand{\bmat}{\begin{bmatrix}}
\newcommand{\emat}{\end{bmatrix}}

\DeclareMathOperator{\trace}{\rm tr}

\DeclareMathOperator{\rank}{rank}

\newcommand{\Rbb}{\mathbb R}

\newcommand{\Cbb}{\mathbb C}

\newcommand{\Sbb}{\mathbb S}
\newcommand{\Tbb}{\mathbb T}
\newcommand{\vb}{\mathbf  v}

\newcommand{\ub}{\mathbf  u}

\newcommand{\Jb}{\mathbf J}

\newcommand{\Ub}{\mathbf U}
\newcommand{\Vb}{\mathbf V}

\newcommand{\Sigmab}{\boldsymbol{\Sigma}}
\newcommand{\Lambdab}{\boldsymbol{\Lambda}}

\newcommand{\zerob}{\boldsymbol{0}}

\DeclareMathOperator{\range}{Range}
\DeclareMathOperator{\kernel}{Ker}
\newcommand{\Hfrak}{\mathfrak{H}}

\newcommand{\Lscr}{\mathscr{L}}

\newcommand{\Hcal}{\mathcal{H}}
\newcommand{\crm}{\mathrm{c}}
\begin{document}

\title{On the Uniqueness Result of Theorem 6 in ``Relative Entropy and the Multivariable Multidimensional Moment Problem''}

\author{Bin~Zhu% <-this % stops a space
\thanks{Manuscript received June 7, 2018; revised January 9, 2019.}% <-this % stops a space
\thanks{This work was funded by the China Scholarship Council (CSC) under file No.~201506230140.}% <-this % stops a space
\thanks{B. Zhu is with the Department of Information Engineering, University of Padova, Via Giovanni Gradenigo, 6b, 35131 Padova, Italy (email: \texttt{zhubin@dei.unipd.it}).}% <-this % stops a space
\thanks{Copyright (c) 2017 IEEE. Personal use of this material is permitted. However, permission to use this material for any other purposes must be obtained from the IEEE by sending a request to \texttt{pubs-permissions@ieee.org}.}
}

\markboth{IEEE Transactions on Information Theory,~Vol.~00, No.~00, Mon.~YYYY}%
{Zhu: On Theorem 6 in ``Relative Entropy and the Multivariable Multidimensional Moment Problem''}

\maketitle

\begin{abstract}
Matrix-valued covariance extension and multivariate spectral estimation are formulated as generalized moment problems in the ``THREE'' approach and its extensions. Under this context, we discuss Theorem 6 in \cite{Georgiou-06} concerning the bijectivity of a moment map defined over a parametric family of spectral densities. In particular, we provide a counterexample in which the moment map under consideration is shown to have a critical point, namely a point at which the Jacobian loses rank. Then with standard techniques in bifurcation theory, we conclude further that the computed critical point is a bifurcation point, which means that the moment map is not injective.
\end{abstract}


%\IEEEpeerreviewmaketitle

\section{Introduction}

% general problem area description, brief review of the literature 

\IEEEPARstart{M}{ultivariate} spectral estimation is an important and challenging problem in the fields of system identification, modeling, and signal processing. The ``THREE'' framework for spectral estimation first appeared in the seminal paper \cite{BGL-THREE-00} by Byrnes, Georgiou, and Lindquist, which can be seen as a generalization of previous works on \emph{rational covariance extension} and \emph{Nevanlinna-Pick interpolation} (cf.~\cite{Kalman,Gthesis,Georgiou-87,Georgiou-87-NP,BLGM-95,byrnes1997partial,BGL98,georgiou1999interpolation,SIGEST-01,BGL-01} and references therein). Since then it has been significantly developed and extended to the multivariate case. We mention an incomplete list of contributions \cite{Georgiou-L-03,georgiou2005solution,PavonF-06,enqvist2008minimal,Z14rat,FRT-11} in the scalar case, and \cite{georgiou2002spectral,Georgiou-02,Georgiou-06,FPR-08,RFP-09,FPZ-10,avventi2011spectral,FMP-12,Z14,Z15,GL-17,baggio2018further,Zhu-Baggio-17,zhu2017parametric} for the multivariate counterpart. In that framework, the steady-state covariance matrix of the output process of a rational filter is used as data for the reconstruction of the input spectrum, which naturally admits a formulation as a generalized moment problem. Due to the typical ill-posedness of moment problems \cite{Grenander_Szego,KreinNudelman}, \emph{entropy}(-like) functionals are then exploited as optimization criteria to promote uniqueness of the solution. More specifically, one tries to find the input spectrum consistent with the output covariance matrix that maximizes some entropy or minimizes some distance index to an \emph{a priori} spectral density. A key feature of the approach is that parameter tuning is allowed in order to achieve high resolution in specified frequency bands.

Different choices of cost functionals lead to different forms of solutions, especially in the multivariate case (cf.~\cite{georgiou2002spectral,Georgiou-06,avventi2011spectral,FMP-12,Z14,Z15}). Among them \cite{Georgiou-06} is an important work utilizing the following relative entropy as the optimization criterion
\[\Sbb(\Phi|\Psi)=\int_\Tbb\trace\left[\Phi\left(\log\Phi-\log\Psi\right)\right]\]
which in turn, draws inspiration from quantum mechanics. Here $\Psi$ is the known prior and $\Tbb$ stands for the unit circle. Minimization of $\Sbb(\Phi|\Psi)$ with respect to $\Phi$ subject to the generalized moment constraint can be worked out explicitly leading to an exponential-type spectral density. Such a solution can also be recovered as a limit case of a family of solutions based on the multivariate Beta divergence discussed in \cite{Z14}. Difficulty arises in the other direction, namely minimization of $\Sbb(\Psi|\,\cdot)$ with respect to the second argument. As reported in \cite{Georgiou-06}, variational analysis and duality reasoning hit an obstruction in the middle because the functional dependence of the optimal primal on the dual variable cannot be described in a closed form (see also \cite{FPR-08}). As a response to this difficulty, Theorem~6 of \cite{Georgiou-06} suggests to ``forgo an explicit form for the entropy functional and start instead with a computable Jacobian''. In other words, a parametric form of the spectral density has been proposed, which possibly does not correspond to any cost functional. Although the statement of that theorem looks rather exciting, it is extremely nontrivial and its validity remains elusive as a rigorous proof is absent. In this note, we are motivated to address this issue. We shall only consider the first half of \cite[Theorem~6]{Georgiou-06} concerning \emph{rational} solutions to the spectral estimation problem.

% open question dicussed in this note and methods used here

The continuation argument is used extensively in the proofs of \cite{Georgiou-06} which follows the previous work \cite{georgiou2005solution} in the scalar case by the same author. As will be reviewed later in Section~\ref{sec:Bifurcat}, in order for the argument to be effective, the Jacobian of the parametric moment map is required to vanish nowhere in the feasible set, which is fulfilled when the prior is taken as $\Psi=\psi I$, namely a scalar spectral density function times the identity matrix. In this work, we show through a numerical example that the requirement of everywhere nonvanishing Jacobian is not met in general by the moment map in question when the prior is nontrivial, contrary to what is claimed in \cite[Section~IV]{Georgiou-06}. Furthermore, a critical point of the moment map is computed in the example and demonstrated to be a bifurcation point.
%, as an application of the bifurcation theory. 
In consequence, the parametric solution to the spectral estimation problem considered in \cite[Section~IV]{Georgiou-06} is generally not unique.

% contents of sections

This note is organized as follows. In Section \ref{sec:Problem}, we review the parametric form of the moment map introduced in \cite{Georgiou-06} that will be the central object of investigation in this work. We give a numerical example in Section \ref{sec:SingJac} where a critical point of the moment map is detected and computed. In Section \ref{sec:Bifurcat}, we apply a part of the bifurcation theory and carry out some further computation which allows us to conclude that the afore obtained critical point is in fact a bifurcation point. Finally, we make some remarks on an alternative parametrization of rational spectral densities.

\section{Problem Review}\label{sec:Problem}

One of the problems considered in \cite{Georgiou-06} is about finding a matrix spectral density function in a particular parametric family that satisfies a (generalized) moment constraint. In order to restate one of the main results of that paper, it is necessary to introduce some notations first:

\begin{itemize}
\item $G(z)=(zI-A)^{-1}B$ is a rational filter defined by the matrix pair $(A,B)$ such that $A\in\Cbb^{n\times n}$ is a stability matrix in the discrete-time sense, $B\in\Cbb^{n\times m}$ has full column rank, and $(A,B)$ is reachable.
\item $\Hfrak_{n}$ is the vector space of $n\times n$ Hermitian matrices over the reals; $\Hfrak_{+,n}\subset\Hfrak_{n}$ contains positive definite matrices.
\item $C(\Tbb;\Hfrak_m)$ is the vector space of continuous $\Hfrak_m$-valued function on the unit circle $\Tbb:=\{z\in\Cbb\,:\,|z|=1\}$. 
\item 
\begin{equation}\label{Gamma}
\Gamma:\,\Phi\mapsto\int G\Phi G^*
\end{equation}
is a linear operator from $C(\Tbb;\Hfrak_m)$ to $\Hfrak_{n}$, where the integral $\int F$ is a shorthand for $\int_{-\pi}^{\pi}F(e^{i\theta})\frac{d\theta}{2\pi}$
and $G^*(z) := B^* (z^{-1}I-A^*)^{-1}$.
The range of $\Gamma$, denoted by $\range\Gamma$, is a subspace of $\Hfrak_{n}$. The symbol used in \cite{Georgiou-06} for the same operator is $L$.
\item $\Lscr_+:=\{\Lambda\in\Hfrak_{n}\;:\;G^*(z)\Lambda G(z)>0,\ \forall z\in\Tbb\}$.
\item $\Psi$ is a bounded and coercive $m\times m$ spectral density function, that is, there exist real positive constants $\mu, M$ such that $\mu I\leq\Psi(e^{i\theta})\leq MI$ for all $\theta\in(-\pi,\pi]$. It admits a (unique) left outer factor $W_\Psi$, namely $\Psi=W_\Psi W_\Psi^*$. The notations used in \cite{Georgiou-06} for $\Psi$ and its factor are $\sigma$ and $\sigma^{1/2}$, respectively.
\item 
\begin{equation}\label{h_map}
h:\,\Lambda\mapsto\int GW_\Psi(G^*\Lambda G)^{-1}W_\Psi^*G^*
\end{equation}
is a map from $\Lscr_+^\Gamma:=\Lscr_+\cap\range\Gamma$ to $\range_+\Gamma:=\Hfrak_{+,n}\cap\range\Gamma$. The domain and codomain of the map are denoted with $\mathcal{K}_+^{\mathrm{dual}}$ and $\mathrm{int}(\mathcal{K})$ in \cite{Georgiou-06}, respectively. Moreover, the argument $\Lambda$ is lowercased in \cite{Georgiou-06}.
\end{itemize}

Theorem 6 of \cite{Georgiou-06} states that the map $h$ is a bijection given any bounded and coercive prior $\Psi$. In other words, given any positive definite matrix $\Sigma\in\range\Gamma$, there exists a unique parameter $\Lambda\in\Lscr_+^\Gamma$ such that the spectral density
\begin{equation}\label{solution_Georgiou}
\Phi=W_\Psi(G^*\Lambda G)^{-1}W_\Psi^*
\end{equation}
solves the generalized moment equation $\Gamma(\Phi)=\Sigma$. A key argument in that paper is that the Jacobian $\nabla h(\Lambda):\range\Gamma\to\range\Gamma$ is invertible for any $\Lambda\in\Lscr_+^\Gamma$. %, where $L(X,Y)$ stands for the space of linear operators between two vector spaces $X$ and $Y$.
We will provide a two-dimensional ($m=2$) numerical counterexample in the next section to this argument showing that the Jacobian of $h$ can be singular at one point. 
% In Section \ref{sec:Bifurcat}, we apply a part of the bifurcation theory and carry out some further computation which allows us to conclude that the afore computed critical point is in fact a bifurcation point. Finally, we make some remarks on an alternative parametrization of rational spectral densities.

\section{Singular Jacobian of the Moment Map}\label{sec:SingJac}

The Jacobian of the moment map $h$, i.e., its Fr\'{e}chet derivative, is a linear operator from $\range\Gamma$ to itself:
%\begin{equation} \label{nabla_h}
%%\begin{split}
%\nabla h(\Lambda):\,\delta\Lambda \mapsto -\int G W_\Psi (G^*\Lambda G)^{-1}(G^*\delta\Lambda G)(G^*\Lambda G)^{-1}W_\Psi^* G^*.
%%\end{split}
%\end{equation}
\begin{equation} \label{nabla_h}
\nabla h(\Lambda):\,\delta\Lambda \mapsto -\int G W_\Psi \Gamma^*(\Lambda)^{-1} \Gamma^*(\delta\Lambda) \Gamma^*(\Lambda)^{-1} W_\Psi^* G^*,
\end{equation}
where $\Gamma^*:X \mapsto G^*XG$ is the adjoint operator of $\Gamma$ in (\ref{Gamma}) from $\Hfrak_{n}$ to $C(\Tbb;\Hfrak_m)$, and $\Gamma^*(\Lambda)^{-1}$ is understood as $(G^*\Lambda G)^{-1}$.

As mentioned in the Introduction, the claim that $\nabla h(\Lambda)$ vanishes nowhere in $\Lscr_+^\Gamma$ is true in the special case when the prior $\Psi=\psi I$ with $\psi$ a scalar spectral density. Details can be found in \cite{Georgiou-06} itself; see also \cite{FPZ-10,zhu2017parametric}. An important observation is that the Jacobian in that case is a self-adjoint operator, and in fact, it is equal to the negative Hessian of a certain cost function. Therefore, the reasoning of nonvanishing Jacobian is built upon the definiteness of the quadratic form $\langle\delta\Lambda,\nabla h(\Lambda)(\delta\Lambda)\rangle$, where the standard inner product over $\Hfrak_{n}$ is defined as $\langle A,B \rangle:=\trace(AB)$. Such reasoning fails in general when $\Psi$ is arbitrarily (but fixed) matrix-valued because the self-adjoint property is lost. One can simply verify that the adjoint operator $\nabla h(\Lambda)^*:\range\Gamma \to \range\Gamma$ of the Jacobian (\ref{nabla_h}) is given by
\begin{equation*}
\delta\Lambda \mapsto -\int G \Gamma^*(\Lambda)^{-1} W_\Psi^* \Gamma^*(\delta\Lambda) W_\Psi \Gamma^*(\Lambda)^{-1} G^*,
\end{equation*}
which is different from $\nabla h(\Lambda)$.

In the sequel, we want to evaluate numerically the Jacobian determinant. Before that, we will have to build a matrix representation of the linear operator $\nabla h(\Lambda)$.

\subsection{Matrix Representation of the Jacobian}

The Jacobian (\ref{nabla_h}) is a linear map from a finite dimensional vector space to itself. It admits a matrix representation if we fix an orthonormal basis of $\range\Gamma$, say $\{\Lambdab_k\}_{k=1}^M$, where $M=m(2n-m)$ in the complex case (cf.~\cite[Proposition 3.1]{FPZ-12} for the dimension). More precisely, the $(j,k)$ element of the \emph{real} $M\times M$ Jacobian matrix $\Jb_h(\Lambda)$ is
\begin{equation}\label{Jac_mat_element}
\langle \Lambdab_j,\nabla h(\Lambda)(\Lambdab_k) \rangle.
\end{equation}
The domain of the map $h$, namely the set $\Lscr_+^\Gamma$, is convex, which is in particular path-connected. We have the next simple proposition.
\begin{proposition}
Consider a $C^1$ map $f:\,D\subset\Rbb^n\to\Rbb^n$ such that $D$ is path-connected. If its Jacobian $\nabla f:\,D\to\Rbb^{n\times n}$ is everywhere nonsingular, then its determinant $\det\nabla f(\cdot)$ does not change sign over $D$.
\end{proposition}
\begin{proof}
Suppose the contrary, i.e., there exist two points $x_1,x_2\in D$ such that $\det\nabla f(x_1)>0$ and $\det\nabla f(x_2)<0$. By the assumption of path-connectedness, there exists a continuous function $p:\,[0,1]\to D$ such that $p(0)=x_1$ and $p(1)=x_2$. Since $f$ is $C^1$, the real-valued function $\det\nabla f(p(\cdot))$ is continuous. By the intermediate value theorem it must be zero for some $t\in(0,1)$.
\end{proof}

Therefore, if a sign change of the Jacobian determinant is detected, the Jacobian of the map under consideration cannot be everywhere nonsingular. This is the idea behind our numerical example.

\subsection{A Numerical Example}

Here we consider the problem of matrix covariance extension of dimension $m=2$ with the maximal covariance lag $p=1$, the (probably) simplest nontrivial case. We have $n=m(p+1)=4$. The matrix pair $(A,B)$ of the filter bank $G(z)$ is given by
\begin{equation*}\label{matrix_pair}
A=\begin{bmatrix}0&I_2\\0&0\end{bmatrix},\quad
B=\begin{bmatrix}0\\I_2\end{bmatrix},\quad\text{with}\quad G(z)=\begin{bmatrix}z^{-2}I_2\\z^{-1}I_2\end{bmatrix}.
\end{equation*}
Let us work in the real case, in which $\range\Gamma$ is the $M=7$-dimensional vector space of symmetric block-Toeplitz matrices of the form
\[\begin{bmatrix}
\Lambda_0 & \Lambda_1^\top\\
\Lambda_1 & \Lambda_0
\end{bmatrix},\]
where $\Lambda_0,\Lambda_1$ are $2 \times 2$ blocks. An orthogonal but unnormalized basis of $\range\Gamma$ can be determined from matrix pairs
\begin{equation}\label{basis_Range_Gamma}
\begin{split}
(\Lambda_0,\Lambda_1)\in\{\zerob\}\times
\left\{\begin{bmatrix}
1&0\\0&0
\end{bmatrix},
\begin{bmatrix}
0&1\\0&0
\end{bmatrix},
\begin{bmatrix}
0&0\\1&0
\end{bmatrix},
\begin{bmatrix}
0&0\\0&1
\end{bmatrix}\right\} \\
\bigcup \left\{\begin{bmatrix}
1&0\\0&0
\end{bmatrix},
\begin{bmatrix}
0&1\\1&0
\end{bmatrix},
\begin{bmatrix}
0&0\\0&1
\end{bmatrix}\right\}
\times \{\zerob\},
\end{split}
\end{equation}
where the bold symbol $\zerob$ denotes the $2\times2$ zero matrix. Normalization of the basis matrices is necessary to compute the quantity (\ref{Jac_mat_element}) correctly.

%Let $\tilde{J}_\kappa(\Lambda)$ be the matrix with elements given in \ref{Jac_mat_element} under the unnormalized basis matrices above. Then the formula for the Jacobian matrix after normalization is simply
%\begin{equation}\label{Jac_mat_normalized}
%\Jb_\kappa(\cdot)=D\tilde{J}_\kappa(\cdot)D,
%\end{equation}
%where $D=\diag\{1/\|\Lambda_k\|\}_{k=1}^M$ is the diagonal scaling matrix. Apparently, the two sides of \ref{Jac_mat_normalized} have the same determinantal sign. Hence we shall work with the unnormalized basis and compute the determinant of $\tilde{J}_\kappa(\cdot)$ in our numerical example.

The prior is taken as $\Psi=KGG^*K^*$, a matrix Laurent polynomial, for
\begin{equation*}
K=\bmat  -0.22 &  -1.23  &  2.22   &   0 \\
   -1.11 &  -0.96 &   1.14  &  2.49 \emat.
\end{equation*}
The polynomial $zKG$ is Schur, with determinantal roots $0.5868,-0.3558$, and thus the outer factor $W_\Psi \equiv zKG$ in this example. %Of course, we have to avoid the degenerate case of $\Psi=\psi I_m$.

The argument $\Lambda$ lives in the open set $\Lscr_+^\Gamma$. In practice, it is better to start with a factor of the form $zCG$ with $C\in\Cbb^{m \times n}$. Then we can form the function
\begin{equation}\label{C_factor_func}
G^* \Lambda G := G^* C^* C G.
\end{equation}
Notice that if we assign the elements of $C$ with Gaussian or uniformly distributed random numbers, it is unlikely that the polynomial $\det zCG$ has roots on the unit circle. From (\ref{C_factor_func}) we have the relation that $\Lambda$ is equal to the projection of $C^*C$ onto the subspace $\range\Gamma$. Details of the spectral factorization (\ref{C_factor_func}) can be found in \cite{FPZ-10,avventi2011spectral,zhu2017parametric}.

We have picked two $C$ matrices with corresponding $\Lambda$ matrices and the determinantal roots of $zCG$ reported below:

\begin{equation*}
C^{(0)}=\bmat -1.08 &  -0.57  &  2.45   &    0 \\
    0.84 &  -0.08  &  1.01  &  0.78 \emat
\end{equation*}
corresponds to the blocks of $\Lambda^{(0)}$
\begin{equation*}
\Lambda_0^{(0)} = \bmat 4.4473  &  0.6681 \\ 
    0.6681  &  0.4698 \emat,\ 
\Lambda_1^{(0)} = \bmat -1.7976 &  -1.4773 \\
    0.6552 &  -0.0624 \emat
\end{equation*}
with the roots of $\det zC^{(0)}G$ at $0.1211 \pm 0.5302i$ (modulus $0.5438$).
\begin{equation*}
C^{(1)}=\bmat 0.63  &  0.67  &  1.45  &    0 \\
    1.68 &  -0.61  &  1.04  &  2 \emat
\end{equation*}
corresponds to the blocks of $\Lambda^{(1)}$
\begin{equation*}
\Lambda_0^{(1)} = \bmat 3.2017 &  0.7387 \\ 
    0.7387  &  2.4105 \emat,\ 
\Lambda_1^{(1)} = \bmat 2.6607  &  0.3371 \\
    3.3600 &  -1.2200 \emat
\end{equation*}
with the roots of $\det zC^{(1)}G$ at $0.7791,-0.6683$.

%\begin{equation}
%\Lambda=\bmat \Lambda_0 & \Lambda_1^\top \\
%\Lambda_1 & \Lambda_0 \emat
%\end{equation}

%Roots of the determinant of $zCG$ are $0.0046\pm0.7815i$ (modulus $0.7816$)  and  $0.6655,0.0986$, respectively for the two arguments $C$ above.

The integral in (\ref{nabla_h}) is approximated with the Riemann sum in Matlab:
\[\int F(\theta)\approx\frac{\Delta\theta}{2\pi}\sum_k F(\theta_k),\]
where $\{\theta_k\}$ are equidistant points on the interval $(-\pi,\pi]$ and the ``step length'' $\Delta\theta=10^{-4}$. With the normalized basis obtained from (\ref{basis_Range_Gamma}), the Jacobian matrix can be computed explicitly as in (\ref{Jac_mat_element}) and its determinant can be evaluated. We have the numerical result $\det \Jb_h(\Lambda^{(k)})=10.6871, -326.6439$ for $k=0,1$, respectively.

Computation of the above example has also been implemented in Mathematica in order to evaluate the integrals symbolically given the numerical values of $\Lambda$. The result is consistent with the numerical computation in Matlab, i.e., a sign change of the Jacobian determinant has been detected. %hence invalidating (one half of) \cite[Theorem 6]{Georgiou-06}, which claims that the Jacobian $\nabla\kappa(\cdot)$ is sign definite.

Further, the critical point $\Lambda^{\mathrm{c}}$ can be computed using the bisection method on the real-valued function $\det \Jb_h(\Lambda^{(t)})$ where
\begin{equation}\label{Lambda_path}
\Lambda^{(t)}=(1-t)\Lambda^{(0)}+t\Lambda^{(1)},\ t\in[0,1]
\end{equation}
is the line segment between $\Lambda^{(k)},\ k=0,1$. We have the blocks
\begin{equation*}
\Lambda^{\mathrm{c}}_{0} = \bmat 4.3901  &  0.6713 \\ 
    0.6713  &  0.5589 \emat,\ 
\Lambda^{\mathrm{c}}_{1} = \bmat -1.5930 &  -1.3940 \\
        0.7793 &  -0.1155 \emat,
\end{equation*}
with the corresponding $t^{\mathrm{c}}=0.0459$, $\det \Jb_h(\Lambda^{\mathrm{c}})=-5.4964\times10^{-14}$, and the two smallest singular values of $\Jb_h(\Lambda^{\mathrm{c}})$ are $1.1053\times10^{-16}$ and $0.0573$. Hence the Jacobian matrix of $h$ loses exactly rank $1$ at $\Lambda^{\mathrm{c}}$.

\section{Characterization of the Critical Point}\label{sec:Bifurcat}

The quest for nowhere vanishing Jacobian is motivated by the use of continuation methods to solve the nonlinear equation $h(\Lambda)=\Sigma$ for the parameter $\Lambda$. The idea is briefly reviewed in the next proposition when the map under consideration is a diffeomorphism (cf.~\cite{allgower1990continuation} for more general settings).

\begin{proposition}
Assume for simplicity that $D,E$ are open and convex subsets of $\Rbb^n$. Let $f:D \to E$ be a $C^{2}$ diffeomorphism. Then for $y \in E$, the solution $x=f^{-1}(y)$ can be found by solving the initial value problem
\begin{equation}\label{IVP}
\left \{
  \begin{aligned}
    \dot{x}(t) & = \left[\,\nabla f\left(x(t)\right)\,\right]^{-1} (y-y_0) \\
    x(0) & = x_0
  \end{aligned} \right.
\end{equation}
and evaluating $x:=x(1)$. The initial value $x_0 \in D$ is arbitrary and $y_0:=f(x_0)$.
\end{proposition}

\begin{proof}
By the assumption of convexity, the line segment
\[p(t)=ty+(1-t)y_0,\ t\in[0,1]\]
is inside $E$. It is easy to verify that the curve $x(t):=f^{-1}\left(p(t)\right)$ solves the IVP (\ref{IVP}). In fact, the differential equation comes from differentiating the two sides of $f\left(x(t)\right)=p(t)$ w.r.t.~$t$ and inverting the Jacobian $\nabla f\left(x(t)\right)$. Due to the assumption that $f$ is a diffeomorphism, the solution curve $x(t)$ is indeed continuously differentiable and the Jacobian of $f$ is everywhere invertible in $D$.
\end{proof}

The precise terminal point $x(1)$ can be obtained using a predictor-corrector algorithm \cite{allgower1990continuation,zhu2018wellposed}. If one is satisfied enough with an approximate solution, then a general-purpose ODE solver can be used to numerically integrate (\ref{IVP}). Of course, the map $f$ in the above proposition being a diffeomorphism is a sufficient condition for the continuation method to return a \emph{unique} solution. This is indeed the case for our $h$ map when the prior takes the special form $\Psi=\psi I$ as mentioned previously (cf.~\cite{zhu2017parametric}). However, in the presence of a singular Jacobian, it can happen that the solution curve to the IVP branches out at a critical point, and several terminal points exist. On the other hand, a numerical ODE solver diverges in that case because the norm of the derivative tends to infinity near the critical point.

Next we shall demonstrate numerically that the critical point computed in Section \ref{sec:SingJac} is a bifurcation point. To this end, it is customary to define the augmented map
\begin{equation*}
\Hcal(\Lambda,t):=h(\Lambda)-p(t)
\end{equation*}
from $\Lscr_+^\Gamma \times [0,1] \to \range\Gamma$, where $p(t):=h\left(\Lambda^{(t)}\right)$ is a smooth curve in $\range_+\Gamma$ with $\Lambda^{(t)}$ in (\ref{Lambda_path}). Under this convention, the curve $\left(\Lambda^{(t)},t\right)$ parametrized by $t$ is in the zero set $\Hcal^{-1}(0)$. When a basis of $\range\Gamma$ is fixed as in the previous section, the map $\Hcal$ can be identified as a function $H$ from a subset of $\Rbb^{M+1}$ to $\Rbb^M$, whose coordinates have the expression
%\begin{equation} %\label{H_coordinate}
%H_k: (x,t) \mapsto \langle \Lambdab_k,\Hcal\left(\Lambda(x),t\right) \rangle, \quad k=1,\dots,M,
%\end{equation}
\begin{equation}\label{H_coordinate}
H_j: (\Lambda,t) \mapsto \langle \Lambdab_j,\Hcal\left(\Lambda,t\right) \rangle, \quad j=1,\dots,M,
\end{equation}
where $\Lambda=\sum_k x_k\Lambdab_k$ with $x\in\Rbb^M$ the coordinate vector. Explicit calls of the coordinate $x$ will be avoided subsequently in order to ease the notation. %unless strictly necessary 

The matrix representation of the augmented Jacobian $\Jb_\Hcal\in\Rbb^{M \times (M+1)}$ can be described in terms of the following vector with each entry in $\range\Gamma\,$:
\begin{equation*}
\nabla H(\Lambda,t)=\left[\begin{array}{ccc|c}
\nabla h(\Lambda)(\Lambdab_1) & \cdots & \nabla h(\Lambda)(\Lambdab_M) & -\dot{p}(t)
\end{array}\right],
\end{equation*}
where $\dot{p}(t)=\nabla h\left(\Lambda^{(t)}\right)\left(\Lambda^{(1)}-\Lambda^{(0)}\right)$. 
%gradient of the component function
%\begin{equation}
%\nabla H_k(x,t)=\left[\begin{array}{c|c}
%\nabla h(\Lambda) & -\dot{p}(t)
%\end{array}\right],
%\end{equation}
Then the $(j,k)$ element of the augmented Jacobian
\begin{equation*}
[\Jb_\Hcal(\Lambda,t)]_{jk}=\langle \Lambdab_j,[\nabla H(\Lambda,t)]_k \rangle.
\end{equation*}
Notice that the last column of $\Jb_\Hcal(\Lambda^\crm,t^\crm)$ does not increase the rank due to the relation $\Lambda^{(t^\crm)}=\Lambda^\crm$. Hence we have
\begin{equation*}
\rank \Jb_\Hcal(\Lambda^\crm,t^\crm)=M-1,\quad \dim\kernel \Jb_\Hcal(\Lambda^\crm,t^\crm)=2.
\end{equation*}

Let us introduce the Lyapunov-Schmidt reduction in our finite dimensional context:
\begin{equation*}
\begin{split}
 & \Rbb^{M+1}=D_1 \oplus D_2,\quad \Rbb^{M}=E_1 \oplus E_2,\textrm{ where} \\
 & D_1:=\kernel \Jb_\Hcal(\Lambda^\crm,t^\crm),\quad D_2:=D_1^\perp, \\
 & E_2:=\range \Jb_\Hcal(\Lambda^\crm,t^\crm),\quad E_1:=E_2^\perp.
\end{split}
\end{equation*}
The above subspaces can be made more precise by performing SVD to the Jacobian matrix of $H$  at $(\Lambda^\crm,t^\crm)$, namely
\begin{equation}\label{SVD_Jac_H}
\begin{split}
\Jb_\Hcal(\Lambda^\crm,t^\crm) & =\Ub \Sigmab \Vb^\top \\
 & =\bmat \ub_{1:M-1} & \ub_M \emat
 \bmat \Sigmab_{M-1} & \zerob \\
 \zerob & \zerob \emat
 \bmat \vb_{1:M-1}^\top \\ \vb_{M:M+1}^\top \emat \\
 & :=\bmat \Ub_1 & \Ub_2 \emat
  \bmat \Sigmab_{M-1} & \zerob \\
  \zerob & \zerob \emat
  \bmat \Vb_{1}^\top \\ \Vb_{2}^\top \emat,
\end{split}
\end{equation}
where $\Sigmab_{M-1}$ is the (square) diagonal matrix containing all the nonzero singular values, $\ub,\vb$ are columns of the orthogonal matrices $\Ub$ and $\Vb$, respectively, and the notation $\ub_{j:k}$ denotes the matrix obtained by putting together the columns $\ub_j,\ub_{j+1},\dots,\ub_k$. It is then elementary to verify that
\begin{equation*}
\begin{split}
 & D_1=\range \Vb_2,\quad D_2=\range \Vb_1, \\
 & E_2=\range \Ub_1,\quad E_1=\range \Ub_2.
\end{split}
\end{equation*}
We can then partition $H$ w.r.t.~the new bases determined by the singular vectors. Specifically, let us define
\begin{equation*}
\tilde{H}(y)=\bmat \tilde{H}_1(y_1,y_2) \\ \tilde{H}_2(y_1,y_2) \emat
:=\bmat \Ub_2^\top \\ \Ub_1^\top \emat H(\Vb_2 y_1 + \Vb_1 y_2),
\end{equation*}
where $y=(y_1,y_2) \in \Rbb^2 \times \Rbb^{M-1}$ are coordinates of the argument vector $(\Lambda,t)$ in (\ref{H_coordinate}) under the new basis. 
%We obviously have $H=\Ub_2\tilde{H}_1+\Ub_1\tilde{H}_2$. 
The Jacobian of $\tilde{H}$ is computed as
\begin{equation*}
\begin{split}
\nabla \tilde{H}(y) & =\bmat \nabla_1 \tilde{H}_1(y_1,y_2) & \nabla_2 \tilde{H}_1(y_1,y_2) \\ 
\nabla_1 \tilde{H}_2(y_1,y_2) & \nabla_2 \tilde{H}_2(y_1,y_2) \emat \\
 & =\bmat \Ub_2^\top \\ \Ub_1^\top \emat \nabla H(\Vb_2y_1+\Vb_1y_2) \bmat \Vb_2 & \Vb_1 \emat,
\end{split}
\end{equation*}
where $\nabla_k \tilde{H_j}$ denotes the Jacobian matrix of $\tilde{H}_j$ w.r.t.~the variable $y_k$. It is then straightforward to check that
\begin{equation*}
\nabla\tilde{H}(y^\crm)=\bmat \zerob & \zerob \\ \zerob & \Sigmab_{M-1} \emat,
\end{equation*}
where $y^\crm$ is the coordinate of $(\Lambda^\crm,t^\crm)$ and $\Sigmab_{M-1}$ nonsingular. Since we have $\tilde{H}_2(y_1^\crm,y_2^\crm)=0$, the implicit function theorem can be applied to assert that locally around $y^\crm$
\begin{equation*}
\tilde{H}_2(y_1,y_2)=0\ \Longleftrightarrow\  y_2=\varphi(y_1)
\end{equation*}
for some smooth function $\varphi$. Substituting $y_2$ with the above local functional dependence on $y_1$ into the equation $\tilde{H}_1(y_1,y_2)=0$, we obtain that equivalently,
\begin{equation*}
b(y_1):=\tilde{H}_1\left(y_1,\varphi(y_1)\right)=0,
\end{equation*}
which is called the \emph{bifurcation equation} at the critical point $y^\crm$ of $\tilde{H}$. Notice that $b$ is a real-valued function defined on some subset of $\Rbb^2$. According to \cite[Definition~8.1.11]{allgower1990continuation}, if the Hessian matrix $\nabla^2 b(y_1^\crm)$ has two eigenvalues of distinct signs, then $y^\crm$ is a \emph{simple} bifurcation point of the equation $\tilde{H}(y)=0$.

Following the derivation in \cite[pp.~77-78]{allgower1990continuation}, we have the equality
\begin{equation*}
\nabla^2 b(y_1^\crm)=\nabla_1^2 \tilde{H}_1(y^\crm).
\end{equation*}
We now need a computable expression for the Hessian matrix. Its operator form is easily obtained
\begin{equation*}
\begin{split}
\nabla_1^2 \tilde{H}_1(y):\  & (\delta y_{1,1},\delta y_{1,2}) \\
 & \mapsto \Ub_2^\top \nabla^2 H(\Vb_2y_1+\Vb_1y_2)(\Vb_2\delta y_{1,1},\Vb_2\delta y_{1,2})
\end{split}
\end{equation*}
as a bilinear map from $\Rbb^2\times\Rbb^2 \to \Rbb$, whose matrix representation follows immediately
\begin{equation}\label{Hessian_mat_b}
\nabla_1^2 \tilde{H}_1(y)= \Vb_2^\top \left[\sum_j u_{jM}\nabla^2 H_j(\Vb_2y_1+\Vb_1y_2)\right] \Vb_2,
\end{equation}
where $\Ub_2\equiv \ub_M$ is the last left singular vector in (\ref{SVD_Jac_H}), and $\nabla^2 H_j$ is the Hessian of the component function in (\ref{H_coordinate}).

Therefore, computation is ultimately reduced to evaluating the $3$-d array of second-order partials $\nabla^2 H$ under the standard (to be normalized) basis introduced in (\ref{basis_Range_Gamma}).
%Define the ``Hessian operator'' (MAYBE OMIT THIS OPERATOR PART)
%\begin{equation}
%\nabla^2 \Hcal(\Lambda,t):=\left[\begin{array}{c|c}
%\nabla^2 h(\Lambda) & 0 \\
%\hline
%0 & -\ddot{p}(t)
%\end{array}\right]
%\end{equation}
Define the symmetric matrix with $\range\Gamma$-valued entries $\nabla^2 H(\Lambda,t):=$
\begin{equation}\label{Hessian_array_H}
\left[\begin{array}{ccc|c}
\nabla^2 h(\Lambda)(\Lambdab_1,\Lambdab_1) & \cdots & \nabla^2 h(\Lambda)(\Lambdab_1,\Lambdab_M) & 0 \\
\vdots & \ddots & \vdots & \vdots \\
\nabla^2 h(\Lambda)(\Lambdab_M,\Lambdab_1) & \cdots & \nabla^2 h(\Lambda)(\Lambdab_M,\Lambdab_M) & 0 \\
\hline
0 & \cdots & 0 & -\ddot{p}(t)
\end{array}\right],
\end{equation}
where
\begin{equation*}
\nabla^2 h(\Lambda)(\delta\Lambda_1,\delta\Lambda_2)=\int F+F^*
\end{equation*}
is the second-order differential of $h$ with
\begin{equation*}
F:=G W_\Psi \Gamma^*(\Lambda)^{-1} \Gamma^*(\delta\Lambda_2) \Gamma^*(\Lambda)^{-1} \Gamma^*(\delta\Lambda_1) \Gamma^*(\Lambda)^{-1} W_\Psi^* G^*
\end{equation*}
and
\begin{equation*}
\ddot{p}(t)=\nabla^2 h(\Lambda^{(t)})(\Lambda^{(1)}-\Lambda^{(0)},\Lambda^{(1)}-\Lambda^{(0)}).
\end{equation*}
The Hessian matrix of the component function results from taking element-wise inner product with (\ref{Hessian_array_H}), i.e., 
%\begin{equation}
%\begin{split}
%[\nabla^2 H_j(\Lambda,t)]_{k\ell} & = \langle \Lambdab_j,[\nabla^2 H(\Lambda,t)]_{k\ell} \rangle, \\
% & \qquad \qquad k,\ell=1,\dots,M+1.
%\end{split}
%\end{equation}
\begin{equation*}
[\nabla^2 H_j(\Lambda,t)]_{k\ell} = \langle \Lambdab_j,[\nabla^2 H(\Lambda,t)]_{k\ell} \rangle,\ k,\ell=1,\dots,M+1.
\end{equation*}

Continuing our numerical example in the previous section, the Hessian matrix $\nabla^2 b(y_1^\crm)$ is computed according to the formula (\ref{Hessian_mat_b}) and its two eigenvalues are $-0.3226, 0.0239$. Therefore, we confirm that $y^\crm$, or equivalently $(\Lambda^\crm,t^\crm)$, is a bifurcation point. Following the very definition of a bifurcation point \cite[p.~76]{allgower1990continuation}, the original map $h$ in (\ref{h_map}) is not injective.

\begin{remark}
The sole purpose of the computation above is to show that the Hessian matrix $\nabla^2 b(y_1^\crm)$ is nonsingular, which according to \cite[p.~78]{allgower1990continuation} is generic. In this case, the Hessian cannot have two eigenvalues of the same sign, since otherwise $\left(\Lambda^{\mathrm{c}},t^{\mathrm{c}}\right)$ would be an isolated zero point of $\Hcal$ which cannot be reached through curve tracing. This is a consequence of a celebrated theorem of Morse \cite[Lemma~8.1.10]{allgower1990continuation}.
\end{remark}

% EXPLANATION OF NOTATION: BASIS VECTORS ARE IN BOLD AS WELL AS MATRIX REPRESENTATIONS OF JACOBIANS AND HESSIANS

\section{Concluding Remarks}

Although only nonvanishing Jacobian is emphasized in \cite{Georgiou-06}, properness\footnote{Recall that a map between two topological spaces is called proper if the preimage of every compact set in the codomain is compact in the domain.} is another important property of the moment map, as it is closely related to the question of surjectivity (cf.~\cite{Zhu-Baggio-17}).
The argument on properness has been made implicitly when the prior is taken to be $\Psi=I$, as can be seen in the second column of \cite[p.~1060]{Georgiou-06}, the part proving that the solution to the IVP can be ``continued'' until $t=1$. However, in the general case of a matrix-valued prior, a proof of the $h$ map being proper does not seem obvious.

It is also worth pointing out that the solution form (\ref{solution_Georgiou}) to the moment problem plays a major role in \cite{takyar2010analytic}, where the factor of $\Psi$ is taken as $W_\Psi=I+KG$ for some $K\in\Cbb^{m\times n}$, which is certainly matrix-valued, i.e., not scalar times identity.

At last, we wish to point out that the problem of real interest to us is how to parametrize (possibly) all rational solutions of ``minimal degree'' to the moment equation $\Gamma(\Phi)=\Sigma$ in the matrix case, since the scalar counterpart has been well solved in \cite{BLGM-95,BGL98,SIGEST-01} in the case of covariance extension. Out of such motive, we would like to mention an alternative parametrization of rational spectral densities discussed in \cite{FPZ-10,Zhu-Baggio-17,zhu2017parametric}, where the ``denominator'' $G^*\Lambda G$ is factored instead of breaking the prior down into factors as in (\ref{solution_Georgiou}). The moment map becomes 
\begin{equation}\label{tau_map}
\tau: C \mapsto \int G (CG)^{-1} \Psi (CG)^{-*} G^*,
\end{equation}
where the parameter $C$ determines the unique right  outer factor of $G^*\Lambda G$ as indicated in (\ref{C_factor_func}). It has been shown in \cite{Zhu-Baggio-17} that the map $\tau$ is \emph{surjective}. Moreover, the derivative of (\ref{tau_map}) can be written down explicitly \cite{zhu2017parametric}, and a singular Jacobian has so far not been detected numerically, which suggests that there is still hope for uniqueness in this alternative parametrization.

%\appendices

\section*{Acknowledgment}

The author would like to thank Dr.~Giacomo Baggio for implementing the numerical example in Mathematica.


\bibliographystyle{IEEEtran}
\bibliography{references}



\begin{IEEEbiography}{Bin Zhu} was born in Changshu, Jiangsu Province, China in 1991. He received the B.Eng.~degree from Xi'an Jiaotong University, Xi'an, China in 2012 and the M.Eng.~degree from Shanghai Jiao Tong University, Shanghai, China in 2015, both in control science and engineering. He is now a Ph.D. student at the Department of Information Engineering, University of Padova, Padova, Italy.

His current research interest includes spectral estimation, rational covariance extension, and ARMA modeling.
\end{IEEEbiography}

\end{document}